\documentclass[12pt]{amsart}
\usepackage{graphicx}
\newtheorem{thm}{Theorem}[section]
\newtheorem{cor}[thm]{Corollary}
\newtheorem{lem}[thm]{Lemma}
\newtheorem{prop}[thm]{Proposition}
\newtheorem{exm}[thm]{Example}
\theoremstyle{definition}
\newtheorem{defn}[thm]{Definition}
\theoremstyle{remark}
\newtheorem{rem}[thm]{Remark}
\numberwithin{equation}{section}

\newcommand{\abs}[1]{\left\vert#1\right\vert}

\def\cp{\mathcal{ P}}

\def\bn{{\mathbb N}}

  \def\G{\Gamma}
\def\d{\delta}

\def\eb{{\mathbf{e}}}
\def\xb{{\mathbf{x}}}
\def\yb{{\mathbf{y}}}
\def\zb{{\mathbf{z}}}

\def\yb{{\mathbf{y}}}

\def\ifn{\infty}

\def\bal1{\textbf{B}_1^+}

\begin{document}

\title[positive solutions]{Positive solution of nonlinear integral equations and surjectivity of Non-Linear Markov Operators On Infinite Dimensional Simplex}

 \author{Farrukh Mukhamedov}
\address{Farrukh Mukhamedov\\
    ,
 Department of Mathematical Sciences, \& College of Science,\\
The United Arab Emirates University, Al Ain, Abu Dhabi,\\
15551, UAE} \email{{\tt far75m@yandex.ru}, {\tt
farrukh.m@uaeu.ac.ae}}

 \author{Otabek Khakimov}
\address{Otabek Khakimov\\
    ,
 Department of Mathematical Sciences, \& College of Science,\\
The United Arab Emirates University, Al Ain, Abu Dhabi,\\
15551, UAE} \email{{\tt hakimovo@mail.ru}, {\tt
otabek.k@uaeu.ac.ae}}

\author{Ahmad Fadillah Embong}
\address{Ahmad Fadillah\\
 Department of Computational \& Theoretical Sciences\\
Faculty of Science, International Islamic University Malaysia\\
Kuantan, Pahang, Malaysia} \email{{\tt ahmadfadillah.90@gmail.com}}

\maketitle

\begin{abstract}
In the present paper, we consider the solvability of positive solutions of nonlinear integral equations by means of investigating
non-linear Markov operators. To solve the problem we find necessary and sufficient condition for the surjectivity of nonlinear Markov operators. \noindent {\it
	Mathematics Subject Classification}: 47H25; 37A30;
47H60
\end{abstract}

\section{Introduction}

Let $\Omega$ be a compact space with a finite Haar measure $\mu$. By $L_1(\Omega)$ and $L_\infty(\Omega)$
we denote the spaces
of all absolutely integrable and essentially bounded functions on $\Omega$, respectively.
Assume that a function $K\in L_1(\Omega^{m+1})$ is defined by
the almost everywhere convergent series:
\begin{equation}\label{yadro}
K(u_1,u_2,\dots,u_m,t)=\sum_{i_1,i_2,\dots i_m\geq1} a_{i_1}(u_1)a_{i_2}(u_2)\dots a_{i_m}(u_m)f_{i_1i_2\dots i_m}(t),
\end{equation}
where $a_{i_j},f_{i_1i_2\dots i_m}\in L_\infty(\Omega)$, $\sup\limits_{t\in\Omega}|f_{i_1i_2\dots i_m}(t)|\leq M$
for any $i_j\in\mathbb N$, $1\leq j\leq m$.

Let us introduce the following integral operator:
\begin{equation}\label{intOp}
{\bf A}x(t)=\int_{\Omega^{m}}K(u_1,u_2,\dots,u_m,t)x(u_1)x(u_2)\dots x(u_m)d\mu(u_1)d\mu(u_2)\dots d\mu(u_m).
\end{equation}

We notice that the domain of operator $A$ may not coincide with whole $L_1$. When $K$ is bounded (or almost everywhere bounded)
it is clear that
$A$ well defined for any $x\in L_1$. In the theory of operators, the main problems are the followings:\\
1) {\it the invariant subset problem}: to finding invariant subset w.r.t. a given operator;\\
2) {\it surjectivity of given operator};\\
3) {\it the existence of fixed points of operator}.

In this paper, we are interested in the studying above-mentioned problems for integral operator \eqref{intOp}
with some conditions on kernel \eqref{yadro}.

First of all, we suppose that the operator \eqref{intOp} is well defined on
$$
\mathcal D=\left\{x\in L_1(\Omega): \sum_{n=1}^\infty\int_{\Omega}a_n(t)x(t)d\mu
\leq1,\ \int_\Omega a_k(t)x(t)d\mu\geq0,\ \forall k\in\mathbb N, \right\}
$$
Then, \eqref{intOp} has the following form
\begin{equation}\label{A22}
{\bf A}x(t)=\sum_{i_1,\dots,i_m\geq1}x_{i_1}\dots x_{i_m}f_{i_1\dots i_m}(t),
\end{equation}
where
$$
x_{i_j}=\int_\Omega a_{i_j}(t)x(t)d\mu,\ \ \ j=\overline{1,m}.
$$
Let us denote
\begin{equation}\label{P_ijk}
P_{i_1i_2\dots i_m,k}=\int_\Omega a_k(t)f_{i_1\dots i_m}(t)d\mu,
\end{equation}
and assume that there exists a positive number $N$ such that
$|P_{i_1\dots i_m,k}|\leq N$, for any $i_1,\dots,i_m,k\in\bn$.
Then, from \eqref{A22} we find
$$
\int_\Omega a_k(t){\bf A}x(t)d\mu=\sum_{i_1,\dots,i_m}P_{i_1\dots i_m,k}x_{i_1}\dots x_{i_m},
$$
or
\begin{equation}\label{birV}
(V(\xb))_k=\sum_{i_1,\dots,i_m}P_{i_1\dots i_m,k}x_{i_1}\dots x_{i_m},\ \ \ \forall \xb\in{\bf B}_1^+,
\end{equation}
where
$$
{\bf B}_1^+=\left\{\xb=(x_1,x_2,\dots):\ x_n\geq0\ \forall n\in\bn,\ \mbox{and } \sum_{k\geq1}x_k\leq1\right\}.
$$

\begin{rem} We note that if hypermatrix $(P_{i_1\dots i_m,k})_{i_1,\dots,i_m,k\geq1}$ is stochastic, i.e.
$$
P_{i_1\dots i_m,k}\geq0,\ \ \ \sum_{n\geq1}P_{i_1\dots i_m,n}=1,\ \ \forall i_1,\dots,i_m,k\in\bn,
$$
then $V({\bf B_1^+})\subset{\bf B_1^+}$ (see Lemma \ref{lem1}). Polynomial operator $V$ defined as \eqref{birV}
is called {\it Polynomial Stochastic Operator} (in short PSO) if $(P_{i_1\dots i_m,k})_{i_1,\dots,i_m,k\geq1}$ is stochastic.
\end{rem}
Thus, the studying of operator \eqref{intOp} on $\mathcal D$ leads us to the studying
of polynomial operator \eqref{birV} on ${\bf B_1^+}$.

The main results of this paper are the followings:

\begin{thm}\label{thm_operA}
Let ${\bf A}$ be an integral operator given by \eqref{intOp} with the kernel \eqref{yadro}.
If hypermatrix $(P_{i_1i_2\dots i_m,k})_{i_1i_2\dots i_m,k\in\mathbb{N}}$ defined as \eqref{P_ijk}
is a stochastic,
then ${\bf A}(\mathcal D)\subset\mathcal D$. Moreover, if  PSO \eqref{birV} is surjective,
then there exists non-empty convex subset $\mathcal D'\subset\mathcal D$ such that
${\bf A}(\mathcal D')=\mathcal D'$.
\end{thm}

\begin{cor}
Let ${\bf A}$ be an integral operator given by \eqref{intOp} with the kernel \eqref{yadro}.
Assume that PSO $V=(P_{ij,k})_{i,j,k\geq1}$ given by \eqref{P_ijk} is a surjective.
Then there exists non-empty convex subset $\mathcal D'\subset\mathcal D$ such that
for any $\varphi\in\mathcal D'$ the integral equation
${\bf A}x=\varphi$
has at least one solution $x$ in $\mathcal D'$.
\end{cor}

\begin{thm}\label{thm_fixed}
Let ${\bf A}$ be an integral operator given by \eqref{intOp} with the kernel \eqref{yadro}.
Assume that $(P_{i_1i_2\dots i_m,k})_{i_1i_2\dots i_m,k\in\mathbb{N}}$ given by \eqref{P_ijk}
be stochastic and $V$ be polynomial operator \eqref{birV}. If $V$ be surjective then
$$
|Fix_{\bf B_1^+}(V)|\leq|Fix_\mathcal{D}({\bf A})|.
$$
Here $Fix_C(B)$ is a set of all fixed points of $B$ on $C$.
\end{thm}

\begin{rem}
According to Theorem \ref{thm_fixed}, the integral operator \eqref{intOp}
may have infinitely many fixed points. Moreover, we notice that
$\mathcal D$ is not compact, therefore,  we can not use the well-known properties
of non-linear integral operators acting on compact sets.
\end{rem}

Natural question arises: does there exist kernel $K\in L_1(\Omega^{m+1})$ given by \eqref{yadro}
such that hypermatrix $(P_{i_1i_2\dots i_m,k})_{i_1i_2\dots i_m,k\in\mathbb{N}}$ defined by \eqref{P_ijk} would be stochastic?
The following example gives positive answer this question.

\begin{exm} Let $\Omega=[0,1]$. We define the following functions
$$
a_n(t)=\left\{
\begin{array}{ll}
2^{-n+1}t^{n-1}, & t\in[0,1),\\[2mm]
0, & t=1.
\end{array}
\right.
$$
and
$$
f_{i_1i_2\dots i_m,k}(t)=\frac{2-t}{2m}\sum_{j\in\{i_1,i_2,\dots,i_m\}}b_{j}(t),
$$
where
$$
b_n(t)=\left\{
\begin{array}{ll}
2^{n}\left(t-\frac{k}{2^{n-1}}\right), & t\in\bigcup\limits_{k=0}^{2^{n-1}-1}\left[\frac{k}{2^{n-1}},\frac{2k+1}{2^n}\right],\\[5mm]
2^{n}\left(\frac{2k+1}{2^{n-1}}-t\right), & t\in\bigcup\limits_{k=0}^{2^{n-1}-1}\left[\frac{2k+1}{2^{n}},\frac{k+1}{2^{n-1}}\right].
\end{array}
\right.
$$
It is clear that $\sup_{t\in\Omega}|f_{i_1i_2\dots i_m,k}(t)|\leq2$. Then we have
$$
0\leq K(u_1,\dots,u_m,t)\leq\left\{
\begin{array}{ll}
2, & \mbox{if }\ \prod_{i=1}^m(u_i-1)\neq0,\\[2mm]
0, & \mbox{if }\ \prod_{i=1}^m(u_i-1)=0.
\end{array}
\right.
$$
Now we are going to check that $(P_{i_1i_2\dots i_m,k})_{i_1i_2\dots i_m,k\in\mathbb{N}}$ is stochastic.

We notice that
$f_{i_1\dots i_m,k}=f_{i_{\pi(1)}\dots i_{\pi(m)},k}$ for any permutation $\pi$ of $(1,2,\dots,m)$.
Non negativity of $a_k$ and $f_{i_1\dots i_m}$ implies that $P_{i_1i_2\dots i_m,k}\geq0$.
Furthermore, we obtain
\begin{eqnarray*}
\sum_{k\geq1}P_{i_1i_2\dots i_m,k}&=&\int_\Omega\sum_{k\geq1}a_k(t)f_{i_1\dots i_m}(t)dt\\
&=&\frac{1}{m}\int_\Omega\sum_{j\in\{i_1,\dots,i_m\}}b_j(t)dt\\
&=&\frac{1}{m}\sum_{j\in\{i_1,\dots,i_m\}}\int_\Omega b_j(t)dt\\
&=&1,
\end{eqnarray*}
which together with $P_{i_1i_2\dots i_m,k}\geq0$ yield that $(P_{i_1i_2\dots i_m,k})_{i_1i_2\dots i_m,k\in\mathbb{N}}$
is stochastic.
\end{exm}

\section{Polynomial stochastic operators}

We recall some necessary notations.
Let $E$ be a subset of $\bn$. For a given $r\geq0$ we denote
$$
S_r^{E} = \left\{\textbf{x}=(x_i)_{i \in E}\in \mathbb{R}^{E}\ : \ x_i\geq 0,\quad \sum\limits_{i\in E}x_i = r\right\},
$$
and
$$
{\bf B}_1^+=\bigcup_{r\in[0,1]}S_r^E.
$$
For the sake of convenience, we always write $S^E$ instead of $S^E_1$.
In what follows, by $\eb_i$ we denote the standard basis in $S^E$,
i.e. $\eb_i=(\d_{ik})_{k\in E}$ ($i\in E$), where $\d_{ij}$ is the
Kroneker delta.

Let $\mathcal P=(P_{i_1\dots i_m,k} )_{i_1,...,i_m,k \in E}$ be a $m+1$-ordered
hypermatrix. An $m+1$-ordered hypermatrix $(P_{i_1\dots i_m,k} )_{i_1,...,i_m,k \in E}$
is called {stochastic} if for any $i_1,i_2,\dots,i_m,k\in E$ it hold
\begin{equation}\label{Pikki}
P_{i_1\dots i_m,k}\geq0,\ \ \ \sum_{n\in E}P_{i_1\dots i_m,n}=1.
\end{equation}
For a given $m+1$-ordered stochastic hypermatrix we define operator $V$ on ${\bf B}_1^+$ as
follows:
\begin{equation}\label{ikkiV}
(V(\xb))_k = \sum\limits_{i_{1},i_2,\dots,i_m\in E}P_{i_1i_2\dots i_m,k}x_{i_1}x_{i_2}\dots x_{i_m}, k\in E.
\end{equation}
Operator \eqref{Pikki},\eqref{ikkiV}
is called
{\it $m$-ordered polynomial stochastic operator} (in short $m$-ordered PSO). Without loss of generality we always assume that
$P_{i_1\dots i_m,k}=P_{i_{\pi(1)}\dots i_{\pi(m)},k}$ for any permutation $\pi$ of $(1,2,\dots,m)$.

\begin{lem}\label{lem1}
Let $V$ be $m$-ordered PSO. Then it hold:
$$
V(S^E)\subset S^E,\ \ \ \ V({\bf B}_1^+\setminus S^E)\subset{\bf B}_1^+\setminus S^E.
$$
\end{lem}
\begin{proof}
Let $V$ be $m$-ordered PSO and $r\in[0,1]$. Take an arbitrary $\xb\in S_r^E$. It is clear that
$(V(\xb))_k\geq0$ for all $k\geq1$. Furthermore, using stochasticity of $\mathcal P$
we find
$$
\sum\limits_{k=1}^\infty(V(\xb))_k=\sum\limits_{i_{1},i_2,\dots,i_m\in E}x_{i_1}x_{i_2}\dots x_{i_m}=r^m,
$$
which implies $V(\xb)\in S_{r^m}^E$. Hence, we infer that
$$
V(S^E)\subset S^E,\ \ \ \ V({\bf B}_1^+\setminus S^E)\subset{\bf B}_1^+\setminus S^E.
$$
The proof is complete.
\end{proof}
The following result plays a crucial role in our further investigation.
\begin{lem}\label{lem2}
Let $V$ be $m$-ordered PSO. Then $V$ is surjective on ${\bf B}_1^+$ iff it is surjective
on $S^E$.
\end{lem}
\begin{proof}
"Only if" part immediately follows from Lemma \ref{lem1}. So, we will prove "if" part.

Assume that $V$ is surjective PSO on $S^E$.
We show $V(S_r^E)=S_{r^m}^E$ for any $r\in[0;1)$. One can see that $V({\bf 0})={\bf 0}$. Due to this fact it is enough to show
$V(S_r^E)=S_{r^m}^E$ for $r\in(0,1)$.

Let $r\in(0,1)$. We define an operator $T_r:S_r^E\to S^E$ as follows $T_r(\xb)=r^{-1}\xb$ for all $\xb\in S_r^E$.
We notice that $T_r$ is bijection. Then, we have
\begin{eqnarray*}
V(S_r^E)&=&V(rT_r(S^E_r))\\
&=&V(rS^E)\\
&=&r^mV(S^E)\\
&=&r^mS^E\\
&=&S_{r^m}^E.
\end{eqnarray*}
Hence, keeping in mind $V({\bf 0})={\bf 0}$, from the last one we infer
$V({\bf B}_1^+\setminus S^E)={\bf B}_1^+\setminus S^E$, which together with $V(S^E)=S^E$ implies
$V({\bf B}_1^+)={\bf B}_1^+$.
The proof is complete.
\end{proof}

\begin{rem}
We notice that the last Lemma can be formulated as follows:
{\it PSO $V$ is surjective on ${\bf B}_1^+$ iff it is surjective
on ${\bf B}_1^+\setminus S^E$.}
\end{rem}

\begin{rem}
Thanks to Lemma \ref{lem1} and lemma \ref{lem2} we may consider PSO $V$ only on $S^E$.
\end{rem}

By \textit{support} of $ \xb = (x_{i})_{i\in E}\in S^E $ we mean a
set $ supp(\xb) = \left\{ i \in E \ : \  x_{i} \neq 0 \right\}. $

We define the face $ \G_{A} $ of $S^{E}$ by setting $ \G_{A} = conv\{\eb_{i} \}_{i \in A} $, here
$conv( B )$ stands for the convex hull of a set $ B $. Let
$$
int \G_{A} = \{ \xb \in \G_A : x_{i}>0,\ \forall \ i \in A \}
$$
be the relative interior of $ \G_{A} $.
Recall that two vectors $ \xb,\yb\in S^E $
are called \textit{orthogonal} (denoted by $ \xb \perp \yb $) if $
supp(\xb) \cap supp(\yb) = \emptyset $. If $ \xb,\yb\in S^E$, then
one can see that  $ \xb \perp \yb $ if and only if $x_ky_k=0$ for all $k\in E$. In what follows, we denote $\xb\circ\yb=\sum_{k\in E}x_ky_k$.

\begin{defn}
	A $m$-ordered PSO $V$ defined on $ S^{E} $ is called \textit{orthogonal
		preserving PSO (OP PSO)} if for any $ \xb,\yb \in S^{E} $ with $ \xb
	\perp \yb $ one has $ V(\xb) \perp V(\yb) $.
\end{defn}
We remind that in \cite{ME18} it was fully described the set of all OP and surjective PSOs when $E\subset\mathbb N$ is finite.

\begin{thm}\cite{ME18}\label{thm_fin}
	Let
	$E=\{1,2,\dots,n\}$ and $V$ be a $m$-ordered PSO on $ S^{E} $. Then the following statements are
	equivalent:
	\begin{itemize}
		\item[(i)] $V$ is orthogonal preserving;
		\item[(ii)] $V$ is surjective;
		\item[(iii)] $V$ satisfy the following conditions:
		\begin{itemize}
			\item[(1)] $V^{-1}(\eb_{i}) = \eb_{\pi(i)} \textmd{ for any } i \in E
			$,
			\item[(2)] $V^{-1}(int \G_{\eb_{i_{1}}\eb_{i_{2}}})=int \G_{\eb_{\pi(i_{1})}\eb_{\pi(i_{2})}}  \textmd{ for any } i_{1},i_{2} \in E, $ \\
			\vdots
			\item[($m$)] $ V^{-1}(int \G_{\eb_{i_{1}} \cdots \eb_{i_{m}}})=int \G_{\eb_{\pi(i_{1})} \cdots \eb_{\pi(i_{m})}} \textmd{ for any } i_{1},\dots, i_{m} \in E, $
		\end{itemize}
		for some permutation $\pi$ of $E$.
	\end{itemize}
\end{thm}

\begin{rem} We notice that if the hypermatrix $\cp$ is given by the
	cubic matrix $(P_{ij,k})$, then the associated PSO reduces to the
	\textit{quadratic stochastic operator (QSO)} given by
	\begin{eqnarray*}
		(V(\xb))_k=\sum\limits_{i,j=1}^\ifn P_{ij,k}x_ix_j, \ \ k\in E.
	\end{eqnarray*}
In \cite{FOA} it was studied OP and surjective QSOs on infinite dimensional simplex.
\end{rem}

\section{Surjectivity and Orthogonal Preservative Polynomial Stochastic Operators}

In this section, we prove an analogous result in Theorem \ref{thm_fin} for infinite dimensional case.
Here, instead of a $m$-ordered PSO $V$ defined on $ S^{E} $ for arbitrary $ E \subset \bn $, we will consider $V$ defined on $ S^{\bn} $ (in short $ S $).

We need auxiliary results to go further in this section.
First, let us describe OP PSO as follows.

\begin{thm}\label{thm_OP_inf}
	Let $V$ be $m$-ordered PSO such that $V(\eb_{i}) = \eb_{i} $ for every $ i \in \bn $. The following statement are equivalent
	\begin{itemize}
		\item[(i)] $V$ is an OP;
		\item[(ii)] For any $i_1,\dots,i_m,k\in\mathbb N$ one has $P_{i_1\dots i_m,k} = 0$ if $k\notin\{i_1,\dots,i_m\}$, ;\\
		\item[(iii)] For any $\xb\in S$ and $k\geq1$ it holds $(V(\xb))_{k} = x_k\mathcal A_k(\xb)$, where
\begin{equation}\label{A_k}
\mathcal A_k(\xb)=x_k^{m-1}+\sum\limits_{j=1}^{m-1}\left(\begin{array}{l}
m\\
j
\end{array}
\right)
x_k^{m-1-j}\sum\limits_{i_{m+1-j},\dots,i_m\neq k}
 P_{k\dots ki_{m+1-j}\dots i_m,k}x_{i_{m+1-j}}\cdots x_{i_{m}}.
 \end{equation}
	\end{itemize}
\end{thm}

\begin{proof}
	We will prove the statement as follow: (i)$\Rightarrow$(ii)$\Rightarrow$(iii)$\Rightarrow$(i).
	
	(i)$\Rightarrow$(ii). Let us assume that $V$ is an OP. Suppose that there exist
$i_1,\dots,i_m\in\mathbb N$ such that $P_{i_1\dots i_m,k}\neq0$ for some $k\notin\{i_1,\dots,i_m\}$.
	Choose a vector $\xb\in S $ as follows
	\begin{eqnarray}\label{eqn_desOP_1}
	\xb= \left( x_{1}, \dots, x_{k-1}, 0, x_{k+1}, \dots \right) \textmd{ such that } x_{i} > 0 \textmd{ for all } i\neq k.
	\end{eqnarray}
It is clear that $\xb\perp\eb_k$. Hence, due to orthogonality preserveness of $V$ one has $V(\xb)\perp V(\eb_k)$.
On the other hand, we have
\begin{eqnarray*}
(V(\xb))_{k}& = &\sum\limits_{j_1,\dots, j_m\in\mathbb N}P_{j_1\dots j_m,k}x_{j_1}\cdots x_{j_{m}}\\
&\geq&P_{i_1\dots i_m,k}x_{i_1}\cdots x_{i_m}>0,
	\end{eqnarray*}
which is contradiction to $V(\xb)\perp V(\eb_k)$. So, we conclude that
for any $i_1,\dots,i_m\in\mathbb N$ it holds $P_{i_1\dots i_m,k}=0$ if $k\notin\{i_1,i_2,\dots,i_m\}$.\\
(ii)$\Rightarrow$(iii). For any $\xb\in S$ and $k\geq1$ we have
\begin{equation}\label{V1234}
(V(\xb))_k=\sum_{i_1,\dots,i_m\in E:\atop{k\in\{i_1,\dots,i_m\}}}P_{i_1\dots i_m,k}x_{i_1}\cdots x_{i_m}.
\end{equation}
Since $\sum_{i=1}^\infty P_{kk\dots k,i}=1$ and $P_{kk\dots k,i}=0$ for all $i\neq k$ one gets $P_{kk\dots k,k}=1$. Keeping in
mind this fact and noting $P_{i_1\dots i_m,k}=P_{i_{\pi(1)}\dots i_{\pi(m)},k}$ for any permutation
$\pi$ of $(1,2,\dots,m)$, from \eqref{V1234} we immediately get (iii).\\
(iii)$\Rightarrow$(i). Pick any pair of orthogonal vectors $\xb,\yb\in S $. Then
by definition we have $x_{k}y_{k} =0$ for all $k\geq1$.

On the other hand,
\begin{eqnarray}
	V(\xb) \circ V(\yb) &=&
	\sum\limits_{k=1}^{\ifn}
	x_{k}\mathcal A_k(\xb)y_{k}\mathcal A_k(\yb)\nonumber \\
	&=&\sum\limits_{k=1}^{\ifn}
	x_{k}y_{k}\mathcal A_k(\xb)\mathcal A_k(\yb)\nonumber \\
	&=& 0 \nonumber
	\end{eqnarray}
	which shows that $V$ is an OP PSO.
\end{proof}

\begin{cor}\label{cor_OP_S_inf}
	Let $V$ be $m$-ordered OP PSO such that $V(\eb_{i}) = \eb_{i} $ for all $ i \in \bn $. Then $V$ satisfy
	\begin{itemize}
		\item[(i)] $ V( \partial S ) \subset \partial S $
		\item[(ii)] $ V( int S ) \subset int S $
	\end{itemize}
\end{cor}

\begin{prop}\label{prop_0}
	Let $V$ be $m$-ordered PSO defined on $S$. For any $ A, B\subset\bn $ one
	has $V(int \G_{A}) \subset int \G_{B}$  if and only if
	$V(\xb^{(0)}) \in int\G_{B}$ for some $\xb^{(0)} \in int\G_{A} $.
\end{prop}

\begin{proof}
	"Only if" part is clear. Let us prove "if" part. First note that for
	any $ \xb \in \G_{A} $ we have
	\begin{equation}\label{S11} (V(\xb))_{k} = \sum\limits_{i_1,\dots,i_m\in\bn} P_{i_1\dots i_m,k}x_{i_{1}}\cdots x_{i_{m}} =
\sum\limits_{i_1,\dots,i_m\in{A}}P_{i_1\dots i_m,k}x_{i_{1}}\cdots x_{i_{m}}, \ \ \ k \in \bn.
	\end{equation}
	Now, we assume that for some $ \xb^{(0)} \in int\G_{A} $ one has
	$V(\xb^{(0)}) \in int\G_{B}$. The last one together with \eqref{S11} mean that for
	any $ k \in B $
	\begin{eqnarray}\label{eqn_sur_prop_iff_A1}
	(V(\xb^{(0)}))_{k} = \sum\limits_{i_1,\dots,i_m\in{A}}P_{i_1\dots i_m,k}x_{i_{1}}^{(0)}\cdots x_{i_{m}}^{(0)} >0,
	\end{eqnarray}
Then noting $\prod_{j=1}^mx_{i_{j}}^{(0)}>0$ for any $i_1,\dots,i_m\in A$, due to \eqref{eqn_sur_prop_iff_A1}
we infer that for any $k\in B$ there exist $i_1(k),i_2(k),\dots,i_m(k)\in A$ such that
$P_{i_1(k)\dots i_m(k),k}>0$. Consequently, for any $\xb\in int\G_A$ one gets
\begin{eqnarray}\label{eqn_sur_prop_iff_A5}
	(V(\xb))_{k}\geq
	P_{i_1(k)\dots i_m(k),k}x_{i_1(k)}\cdots x_{i_{m}(k)} > 0,\ \ \textmd{ for any } k\in B
	\end{eqnarray}

On the other hand, from \eqref{S11} one gets
	\begin{eqnarray}\label{eqn_sur_prop_iff_A2}
	(V(\xb^{(0)}))_{\widetilde{k}} = \sum\limits_{i_1,\dots,i_m\in{A}}P_{i_1\dots i_m,\widetilde{k}}x_{i_{1}}^{(0)}\cdots x_{i_{m}}^{(0)}=0,\ \ \
\mbox{if } \widetilde{k} \notin B.
	\end{eqnarray}	
Due to \eqref{eqn_sur_prop_iff_A2}, one infers that
$P_{i_1\dots i_m,\widetilde{k}} = 0$ for all $i_1,\dots,i_m\in A$, $\widetilde{k}\notin B$.
Keeping in mind that fact, for any  $\xb \in int\G_{A} $, we obtain
	\begin{eqnarray}\label{eqn_sur_prop_iff_A6}
	(V(\xb))_{\widetilde{k}} = \sum\limits_{i_1,\dots,i_m\in A}
	P_{i_1 \dots i_m,\widetilde{k}}x_{i_{1}}\cdots x_{i_{m}} =0,\ \ \ \textmd{ for any } \widetilde{k}\notin B.
	\end{eqnarray}
	Hence, from \eqref{eqn_sur_prop_iff_A5}, \eqref{eqn_sur_prop_iff_A6}
	it follows  that  $ V(\xb) \in int\G_{B} $. This completes the
	proof.
\end{proof}

\begin{rem}\label{rem_Cauchy_pro}
	Let us recall Cauchy Product of the following form:
$$
\left( \sum\limits_{i=1}^{\ifn} x_{i} \right)^{m} = \sum\limits_{i_1,\dots,i_m\in \bn}x_{i_{1}} \cdots x_{i_{m}},\ \ \ m\in\mathbb N,
$$
where $\sum_{i=1}^{\ifn} x_{i} < \ifn$.
\end{rem}

\begin{prop}\label{prop_1}
	Let $V$ be a $m$-ordered PSO with $ V^{-1} (\eb_{k}) $ is non-empty for some $k \in\bn $.
Let $\xb\in V^{-1} (\eb_{k}) $, then for any $ \yb \in \G_{supp({\xb})} $ one has
$$
V(\yb) = \eb_{k}
$$
\end{prop}

\begin{proof}
	Denote that $ A =  supp({\xb})  $.
	Without the loss of generality, we suppose that $ \abs{ A} > 1 $ since if
$ \abs{A} = 1 $, then $ \G_{supp({\xb})} = \eb_{supp({\xb})} = {\xb} $ which is clearly $ V({\xb}) = \eb_{k} $.
	Let $ {\xb} \in V^{-1}(\eb_{k}) $, then one has
	\begin{eqnarray}\label{eqn_prop_1}
	(V({\xb}))_{n} = \left\{
	\begin{array}{lll}
	\sum\limits_{i_1,\dots,i_m\in A} P_{i_1\dots i_m,n}{x}_{i_{1}}\cdots {x}_{i_{m}} =0 & \mbox{if } & n \neq k,\\[4mm]
	\sum\limits_{i_1,\dots,i_m\in A} P_{i_1\dots i_m,n}{x}_{i_{1}}\cdots{x}_{i_{m}} =1 & \mbox{if } & n = k.
	\end{array}
	\right.
	\end{eqnarray}
	Assume that, $ P_{\tilde{i}_1\dots\tilde{i}_m,k }< 1 $ for some $\tilde{i}_1,\dots,\tilde{i}_m\in A$.
	Keeping in mind Remark \ref{rem_Cauchy_pro}, then one has
	\begin{eqnarray*}
	(V({\xb}))_{k}&=&\sum\limits_{i_1,\dots,i_m\in A} P_{i_1\dots i_m,n}{x}_{i_{1}}\cdots{x}_{i_{m}}\\
&<& \sum\limits_{i_1,\dots,i_m\in A}{x}_{i_{1}}\cdots {x}_{i_{m}}\\
& =& \left( \sum\limits_{i \in A}x_{i}\right)^{m}\\
& =& 1,
	\end{eqnarray*}
	which contradicts to \eqref{eqn_prop_1}. Therefore we conclude that
	\begin{eqnarray}\label{eqn_prop_1_1}
	P_{i_1\dots i_m,k} = 1 \textmd{ for all } i_1,\dots,i_m\in A
	\end{eqnarray}
	According to Proposition \ref{prop_0}, then one has
$$
V\left( int \G_{A} \right)= \eb_{k}.
$$
In order to complete the proof we have to show
$V(\partial\G_A)=\eb_k$. Take any arbitrary $\yb\in \partial \G_{A} $.
Then there is $ \tilde{A} \subset A $ such that $ y_{i}=0 $ for all $ i \in \tilde{A} $.
Clearly $ A \backslash \tilde{A} \neq \emptyset $ and $ \sum\limits_{i \in A \backslash \tilde{A}}y_{i} = 1 $.
Taking into account \eqref{eqn_prop_1_1}, then one finds
	\begin{eqnarray*}
	(V(\yb))_{k} &=& \sum\limits_{i_1,\dots,i_m\in {A \backslash \tilde{A}}} P_{i_1\dots i_m,k}y_{i_{1}}\cdots y_{i_{m}}\\
& =&
	\sum\limits_{i_1,\dots,i_m\in {A \backslash \tilde{A}}} y_{i_{1}}\cdots y_{i_{m}}\\
&=& 1,
	\end{eqnarray*}
	which yields that $V(\yb)=\eb_{k} $. This completes the proof.
	
\end{proof}

Now we are ready to prove main result in this section.
\begin{thm}\label{thm_sur_op}
	Let
	$V$ be a $m$-ordered PSO such that $V(\eb_{i}) = \eb_{i} $ for all $ i\in \bn $. Then the following statements are
	equivalent:
	\begin{itemize}
		\item[(i)] $V$ is orthogonal preserving;
		\item[(ii)] $V$ is surjective;
		\item[(iii)] $V$ satisfy the following conditions:
		\begin{itemize}
			\item[(1)] $V^{-1}(\eb_{i}) = \eb_{i} \textmd{ for any } i \in \bn
			$,
			\item[(2)] $V^{-1}(int \G_{\eb_{i_{1}}\eb_{i_{2}}})=int \G_{\eb_{i_{1}}\eb_{i_{2}}}  \textmd{ for any } i_{1},i_{2} \in \bn, $ \\
			\vdots
			\item[($ m $)] $V^{-1}(int \G_{\eb_{i_{1}} \cdots \eb_{i_{m}}})=int \G_{\eb_{i_{1}} \cdots \eb_{i_{m}}} \textmd{ for any } i_{1},\dots, i_{m} \in \bn, $
		\end{itemize}
	\end{itemize}
\end{thm}
\begin{proof}
	(i)$\Rightarrow$(ii). Denote
	\begin{eqnarray*}
		&& A_1=\{ [1,n]\subset \bn : n\in \bn \}, \\
		&&A_2=\{ a\subset [1,n] : |[1,n]\setminus a|\geq 2,  n\in \bn \}, \\
		&&A_3=\{ b\subset \bn : a\subset b, \ a\in A_1\cup A_2,
		|\bn\setminus b|<\infty,  \},\\
		&&A=A_1\cup A_2\cup A_3.
	\end{eqnarray*}
	Let us introduce an order in $A$ by inclusion, i.e.  $a\leq b$ means
	that $a\subset b$ for $a,b\in A$. It is clear that $A$ is a
	completely ordering set. To prove that $V$ is surjective we will use
	transfer induction method with respect to the set $A$. Obviously, for the first element $\{1\}$ of the set  $A$, the operator $V$
	on $S^{\{1\}}$ is surjective (see \cite{ME18}).
	Taking into account Corollary \ref{cor_OP_S_inf}, then we have
$$
V(S^{\{1\}}) = S^{\{1\}}
$$
	
	Assume that for an
	element $a\in A$ the operator $V$ is surjective on $S^b$ for every
	$b<a$. Let us show that it is surjective on $S^a$. Suppose that
	$V(S^a)\neq S^a$. For the boundary  $\partial S^a$ of $S^a$ we have
	$\partial S^a=\bigcup\limits_{c\in A: c<a}S^c$. According to the
	assumption of induction one gets
	\begin{equation}\label{5eq3}
	V(\partial S^{a})=\partial S^{a}.
	\end{equation}
	On the other hand,  there exist $\xb,\yb\in int S^a$ such that $\xb\in
	V(S^a), \ \yb\notin V(S^a)$. The segment $[\xb,\yb]$ contains at least
	one boundary point $\zb$ of the set $V(S^a)$. Since $V:S^a\to
	V(S^a)$ is continuous, then the boundary point goes to
	boundary one. Therefore, for $\zb\in ri S^a$ one has $V^{-1}(\zb)\in
	\partial S^a$, which contradicts to \eqref{5eq3}.
	Thus, $V$ is surjective.\\
	(ii)$\Rightarrow$(iii). Assume that $V$ is surjective and
	let $V^{-1}(\eb_i) $ be the preimage of $ \eb_{i} $.
		Suppose that $ |supp(V^{-1}(\eb_{i}))| >1  $, then there is $ \yb \in V^{-1}(\eb_{i}) $ such that $ \abs{supp(\yb)} > 1 $. Hence according to Proposition \ref{prop_1} for any $ \xb \in \G_{supp(\yb)} $ one has
$$
V(\xb) = \eb_{i}.
$$
	Due to $ \abs{supp(\yb)} > 1 $, then there is $ i_{0} \neq i $ such that $ \eb_{i_{0}} \in \G_{supp(\yb)} $ i.e., $V(\eb_{i_{0}}) = \eb_{i} $ which leads to a contradiction to early assumption in the theorem.
	Thus $ \abs{supp(V^{-1}(\eb_{i}))} =1 $. Using the same argument as before there is unique $ i $ such that $V(\eb_{i})  $ maps to $ \eb_{i} $.
	Hence we obtain (iii)(1).
	
	Further, let $ j \in \{ 2,\dots,m \} $. Take $ \yb \in
	int\G_{\eb_{i_{1}}\cdots \eb_{i_{j}}} $ and let $ \xb \in
	V^{-1}(\yb) $. Using Proposition \ref{prop_0} we have
$$
V(int\G_{supp(\xb)} ) \subset int\G_{\eb_{i_{1}} \cdots \eb_{i_{j}}}.
$$
	In fact, we have
	$$
supp(\xb) = \{ i_{1},\dots, i_{j} \} \ \textmd{ for any } \xb \in V^{-1}(\yb).
$$
	If not, then there exists $ j' \in supp(\xb) \backslash \{
	i_{1},\dots, i_{j} \} \neq \emptyset $. Then $V(\eb_{j'}) \in
	V(int\G_{supp(\xb)}) \subset \G_{\eb_{i_{1}}\cdots \eb_{i_{j}}} $,
	which is a contradiction. Therefore we obtain (iii)($j$) for any $ j
	\in \{ 2,\dots,m \} $.\\
(iii)$\Rightarrow$(i). Take arbitrary $ A \subset \bn $ such that $ \abs{A} \leq m $. According to (iii), then one has
	$$
V(int \G_{A}) = int \G_{A}
$$
	Therefore for any $ k \notin A $ and $ \xb \in int \G_{A} $ one has
$$
(V(\xb))_{k} = \sum\limits_{i_1,\dots,i_m\in A} P_{i_1\dots i_m,k}x_{i_{1}} \cdots x_{i_{m}} = 0.
$$
	Taking into account $ \xb \in int \G_{A} $ one infers that
$$
P_{i_1\dots i_m,k} = 0,\ \ \  \forall i_1,\dots,i_m\in A,\ \forall k\notin A.
$$
	Due to arbitrariness of $ A \subset \bn $ yields
$$
P_{i_1\dots i_m,k} = 0,\ \ \ \forall i_1,\dots,i_m\in\bn,\ \forall k\notin\{i_1,\dots,i_m\}.
$$
	Then, according to Theorem \ref{thm_OP_inf}, $V$ is an OP.
	This completes the proof of theorem.
\end{proof}

Immediately, from the last theorem one concludes the following
corollary.

\begin{cor}
		Let
		$V$ be a $m$-ordered PSO such that $V(\eb_{i}) = \eb_{\pi(i)} $ for all $ i\in \bn $. Then the following statements are
		equivalent:
		\begin{itemize}
			\item[(i)] $V$ is orthogonal preserving;
			\item[(ii)] $V$ is surjective;
			\item[(iii)] $V$ satisfy the following conditions:
			\begin{itemize}
				\item[(1)] $V^{-1}(\eb_{i}) = \eb_{\pi(i)} \textmd{ for any } i \in \bn
				$,
				\item[(2)] $ V^{-1}(int \G_{\eb_{i_{1}}\eb_{i_{2}}})=int \G_{\eb_{\pi(i_{1})}\eb_{\pi(i_{2})}}  \textmd{ for any } i_{1},i_{2} \in \bn, $ \\
				\vdots
				\item[($m$)] $ V^{-1}(int \G_{\eb_{i_{1}} \cdots \eb_{i_{m}}})=int \G_{\eb_{\pi(i_{1})} \cdots \eb_{\pi(i_{m})}} \textmd{ for any } i_{1},\dots, i_{m} \in \bn, $
			\end{itemize}
			for some permutation $\pi$ of $\bn$.
		\end{itemize}
	
\end{cor}

\section{Surjectivity of PSO}

The natural question may arise as follows: Is there any surjective PSO without the condition $V(\eb_{i}) = \eb_{\pi(i)} $? This question is answered by the following example.

\begin{exm} \label{exm_not_op_sur} Let $\overline{V}$
 be a surjective $m$-ordered PSO with $\overline{V}(\eb_i)=\eb_i$, $\forall i\in\bn$. We define new $m$-ordered PSO as follows:
	\begin{equation}\label{V121}
	(V(\xb))_{k} = \left\{
	\begin{array}{lll}
		(\overline{V}(\xb))_1+(\overline{V}(\xb))_2, & \mbox{if } k=1,\\[4mm]
		(\overline{V}(\xb))_{k+1},  & \mbox{if }  k>1.
	\end{array}
	\right.
	\end{equation}
	One can see that $ V(\eb_{1}) = V(\eb_{2}) = \eb_{1} $, which implies that $V$ is not OP. We show
that $V$ is surjective.
	
Take any $\yb=(y_1,y_2,\dots)\in S$. Then, $\yb'=(0,y_1,y_2,\dots)$ also belongs to $S$. Since surjectivity of $\overline{V}$
one can find $\xb=(x_1,x_2,\dots)\in S$ such that $\overline{V}(\xb)=\yb'$. We notice that, due to Theorem \ref{thm_sur_op},
$\overline{V}$ is OP. Hence, thanks to Theorem \ref{thm_OP_inf} (iii), one can see that $x_1=0$. Then, from \eqref{V121}
we obtain
$$
(V(\xb))_k = (\overline{V}(\xb))_{k+1},\ \ \ k\in\bn,
$$
which implies $V(\xb)=\yb$. Finally, arbitrariness of $\yb\in S$ yields surjectivity of $V$.
\end{exm}

The considered Example \ref{exm_not_op_sur} naturally leads our attention to the description of the set
of all surjective PSOs on an infinite dimensional simplex.
In this section, we want to describe necessity and sufficiently condition to PSO to be surjective

\begin{thm}\label{prop_2}
	Let $V$ be a surjective $m$-ordered PSO on $S$. Then there exists a
	sequence
	$\{j_k\}_{k\geq1}\subset\mathbb N$
	such that $P_{j_kj_k\dots j_k,k}=1$ for all $k\in\bn$.
\end{thm}

\begin{proof}
	Denote,
$$
I_{k} = \left\{ j \in \bn:\ P_{j\cdots j, k} = 1 \right\}
$$
	First, we will show that for any surjective $m$-ordered PSO $V$, the set $ I_{k} $ is non-empty for any $ k \in \bn $.
	Due to surjectivity of $V$, for any $ \eb_{k} $, $ k\in \bn $ there is $ \xb^{(k)} \in S $ such that
	\begin{eqnarray}\label{eqn_prop_2}
	V(\xb^{(k)}) = \eb_{k}
	\end{eqnarray}
	Now, we divide into two cases.
	
	\textbf{Case 1:} Let $ \abs{supp(\xb^{(k)})} = 1 $ and we denote $ supp(\xb^{(k)}) = j_{k} $, therefore
$$
(V(\xb^{(k)}))_{k} = P_{j_{k}\cdots j_{k},k} x_{j_{k}}^k =1.
$$
	Clearly $ P_{j_{k}\cdots j_{k},k} =1 $, hence $ j_{k} \in I_{k} $
	
	\textbf{Case 2:} Let $|supp(\xb^{(k)})|>1$ and let $ A = supp(\xb^{(k)}) $. From \eqref{eqn_prop_2} we get
	\begin{eqnarray}
	(V(\xb^{(k)}))_{k} = \sum\limits_{i_1,\dots,i_m\in A}P_{i_1\dots i_m,k}x_{i_{1}}\cdots x_{i_{m}} =1.
	\end{eqnarray}
	Now suppose that there is $ \bar{I} = \{ \bar{i}_{{1}}, \dots, \bar{i}_{{m}} \} \in A^m$ such that $ P_{\bar{i}_1\dots\bar{i}_m,k} < 1 $.
	Keeping in mind Remark \ref{rem_Cauchy_pro}, one has the following
	\begin{eqnarray}
	(V(\xb^{(k)}))_{k} &=& \sum\limits_{i_1,\dots,i_m\in A}P_{i_1\dots i_m,k}x_{i_{1}}\cdots x_{i_{m}} \nonumber \\
	&\leq& \sum\limits_{\{i_1,\dots i_m\}\in A\setminus\bar{I}}x_{i_{1}}\cdots x_{i_{m}} + P_{\bar{i}_1\dots\bar{i}_m,k}x_{\bar{i}_{1}}\cdots x_{\bar{i}_{m}} \nonumber \\
	&<& \sum\limits_{i_1,\dots,i_m\in A}x_{i_{1}}\cdots x_{i_{m}}
	= \left(  \sum\limits_{i \in A}x_{i} \right)^{m} = 1 \nonumber
	\end{eqnarray}
	The last statement contradict to \eqref{eqn_prop_2}, therefore we conclude that
$$
P_{i_1\dots i_m,k} = 1 \textmd{ for all } i_1,\dots,i_m\in{A}.
$$
	Thus, $ P_{i\cdots i,k} =1 $ for any $ i \in A $. Consequently, $A \subset I_{k}$. This yields that $I_k\neq\emptyset$
for any $k\geq1$.
	
	So, knowing non-emptiness of $I_k$ for all $k\in\bn$, we define a sequence $\{j_k\}_{k\geq1}$ as follows $j_k=\inf I_k$. Hence, due to
construction of $I_k$ we get
$P_{j_kj_k\dots j_k,k}=1$ for all $k\in\bn$. This completes the proof.
\end{proof}

\begin{thm}\label{thm_surject}
	Let $V$ be a $m$-ordered PSO on $S$. Assume that there exists a
sequence
$\{j_n\}_{n\geq1}\subset\mathbb N$ such that for any $\xb\in S$ with $supp(\xb)\subset\{j_n\}_{n\geq1},\ \ n\in\mathbb N$ one has
\begin{equation}\label{nec1}
(V(\xb))_{k} = x_{j_k}\sum\limits_{j=0}^{m-1}\left(\begin{array}{l}
m\\
j
\end{array}
\right)
x_{j_k}^{m-1-j}\sum\limits_{i_{m+1-j},\dots,i_m\neq j_k}
 P_{j_k\dots j_ki_{m+1-j}\dots i_m,k}x_{i_{m+1-j}}\cdots x_{i_{m}},
\end{equation}
\begin{equation}\label{nec2}
P_{j_k\dots j_k,k}=1,\ \ \ \ \forall k\in\bn.
\end{equation}
Then $V$ is surjective.
\end{thm}

\begin{proof}
Let us assume that there exists a set of indexes $I=\{j_n\}_{n\geq1}$ which satisfies
	\eqref{nec1},\eqref{nec2}. We define new hypermatrix $\tilde{\mathcal P}=(\tilde{P}_{i_1\dots i_m,k})_{i_1,\dots,i_m,k\geq1}$
as follows
$$
\tilde{P}_{i_1\dots i_m,k}=\left\{
\begin{array}{ll}
P_{\alpha(i_1)\dots\alpha(i_m),k}, & \alpha(i_1),\dots,\alpha(i_m)\in\{j_n\}_{n\geq1},\\
0, & \mbox{otherwise},
\end{array}
\right.
$$
where $\alpha(k)=j_k$ for all $k\geq1$. One can check that $\tilde{\mathcal P}$ satisfies the following
conditions:
$$
\tilde{P}_{i_1\dots i_m,k}=0,\ \ \ \forall k\notin\{i_1,\dots,i_m\}.
$$
Then, for the $m$-ordered PSO $\tilde{V}$ is given by
$$
(\tilde{V}(\xb))_k=\sum_{i_1,\dots,i_m\in\bn}\tilde{P}_{i_1\dots i_m,k}x_{i_1}\cdots x_{i_m},\ \ \ \forall\xb\in S,
$$
holds $\tilde{V}(\eb_k)=\eb_k$ for all $k\geq1$. Thanks to Theorem \ref{thm_OP_inf}, PSO $\tilde{V}$ is OP.
Hence, due to Theorem \ref{thm_sur_op}, PSO $\tilde{V}$ is surjective.

On the other hand, if we denote $S^I=\{\xb\in S: supp(x)\subset I\}$ then
operator $T:S\to S^I$ given by $T(\xb)_k=x_{\alpha(k)}$ is bijection.
Furthermore, we have
$\tilde{V}=V\circ T$. Then, keeping in mind that $\tilde{V}$ is surjective and $T$ is bijection,
we infer that $V=\tilde{V}\circ T^{-1}$ is surjective.
	The proof is completed.
\end{proof}

\begin{thm}\label{last_thm}
Let $V$ be surjective $m$-ordered PSO. If $V$ is OP then there exists permutation $\pi$ of natural numbers
such that for any $i_1,\dots,i_m,k\in\bn$
\begin{equation}\label{sdgf}
P_{i_1\dots i_m,k}=0,\ \ \mbox{if}\ \pi(k)\notin\{i_1,\dots,i_m\}.
\end{equation}
\end{thm}

\begin{proof}
Let $V$ be surjective. Then, according to Theorem \ref{prop_2} there exists a sequence
$\{j_k\}_{k\geq1}$ which satisfies the followings:
$$
P_{j_k\dots j_k,k}=1,\ \ \forall k\in\bn.
$$
Now we show that $\{j_k\}_{k\geq1}=\bn$. Suppose that $\bn\setminus\{j_k\}_{k\geq1}\neq\emptyset$. This means
that
there exist $k_0\in\bn$ and $i_0\in\bn\setminus\{j_k\}_{k\geq1}$ such that
$P_{i_0\dots i_0,k_0}\neq0$. Then, we obtain
$$
(V(\eb_{i_0}))_{k_0}(V(\eb_{j_{k_0}}))_{k_0}\neq0,
$$
which contradicts orthogonally preserveness of $V$. Hence, we infer that
$\{j_k\}_{k\geq1}=\bn$, i.e. $\pi(k)=j_k$ is permutation of $\bn$.
Consequently, we have
\begin{equation}\label{vet}
V(\eb_{\pi(k)})=\eb_k,\ \ \forall k\in\bn.
\end{equation}
We consider the following $m$-ordered PSO:
$$
(\tilde{V}(\xb))_k=\sum_{i_1,\dots,i_m\in\bn}\tilde{P}_{i_1\dots i_m,k}x_{i_1}\cdots x_{i_m},\ \ \xb\in S,
$$
where
\begin{equation}\label{asdf}
\tilde{P}_{i_1\dots i_m,k}=P_{\pi(i_1)\dots \pi(i_m),k}.
\end{equation}
It is clear that
\begin{equation}\label{fgr}
\tilde{V}(\xb)=V(\xb_{\pi}),\ \ \ \forall \xb\in S,
\end{equation}
where $\xb_\pi=(x_{\pi(1)},x_{\pi(2)},\dots)$.
Then, surjectivity of $V$ implies surjectivity of $\tilde{V}$.

On the other hand, from \eqref{vet},\eqref{fgr} one gets
$$
\tilde{V}(\eb_k)=\eb_{k},\ \ \forall k\in\bn.
$$
Surjectiveness of $\tilde{V}$ together with the last one, due to Theorem \ref{thm_sur_op}
yields that $\tilde{V}$ is OP.
Then $\tilde{V}$ satisfies all conditions of Theorem \ref{thm_OP_inf}. Hence, for any $i_1,\dots,i_m,k\in\bn$
it holds
\begin{equation}\label{asdfg}
\tilde{P}_{i_1\dots i_m,k}=0,\ \ \mbox{if}\ k\notin\{i_1,\dots,i_m\}.
\end{equation}
Keeping in mind $n=\pi(\pi^{-1}(n))$, $n\in\bn$, from \eqref{asdf},\eqref{asdfg} we find
$$
P_{i_1\dots i_m,k}=0,\ \ \ \mbox{if }\ k\notin\left\{\pi^{-1}(i_1),\dots,\pi^{-1}(i_m)\right\},
$$
which is equivalent to \eqref{sdgf}. The proof is complete.
\end{proof}

\section{Proofs of the Main Theorems}

In this section, we are going to prove the Main Theorems which are formulated in Section 1.

\begin{proof}[Proof of Theorem \ref{thm_operA}]
Let us assume that $\left(P_{ij,k}\right)_{i,j,k\geq1}$ given by \eqref{P_ijk} is stochastic.
Then $f_{i_1\dots i_m}\in\mathcal D$, for all $i_1,\dots,i_m\in\bn$.

Now we show that ${\bf A}x\in\mathcal D$ for any $x\in\mathcal D$.
Let us define an operator
$\Im:\mathcal D\to{\bf B}_1^+$ as follows:
\begin{equation}\label{T}
\Im(x)_k=\int_\Omega a_k(t)x(t)dt,\ \ k=1,2,3,\dots,\ \ \forall x\in\mathcal D.
\end{equation}
From \eqref{intOp} we obtain
\begin{equation}\label{A12}
{\bf A}x(t)=\sum_{i_1,\dots,i_m=1}^\infty\Im(x)_{i_1}\dots \Im(x)_{i_m}{f}_{i_1\dots i_m}(t).
\end{equation}
Since the set $\{f_{i_1\dots i_m}\}_{i_1,\dots,i_m\geq1}$ is uniformly bounded and $\Im(x)\in{\bf B}_1^+$
we infer that
${\bf A}x$ is bounded, hence $Ax\in L_1(\Omega)$.
Furthermore, noting \eqref{P_ijk} from \eqref{A12} we obtain
\begin{equation}\label{x_k}
\Im({\bf A}x)_k=\sum\limits_{i_1,\dots,i_m=1}^\infty P_{i_1\dots i_m,k}\Im(x)_{i_1}\cdots\Im(x)_{i_m}, \ \ k=1,2,3\dots
\end{equation}
Hence, one has $\Im({\bf A}x)_k\geq0$, $\forall k\in\bn$
and
\begin{eqnarray*}
\sum_{n=1}^\infty\Im({\bf A}x)_n&=&\sum_{i_1,\dots,i_m\in\bn}\Im(x)_{i_1}\cdots\Im(x)_{i_m}\\
&=&\left(\sum_{i=1}^\infty\Im(x)_i\right)^m\\
&\leq&1.
\end{eqnarray*}
The last one together with
${\bf A}x\in L_1(\Omega)$ implies ${\bf A}x\in\mathcal D$.

Now suppose that stochastic hypermatrix $(P_{i_1\dots i_m,k})$ defines surjective
operator $V$ given by \eqref{birV}. Then, according to Lemma \ref{lem2}, $V$ is surjective
on $S$. Thanks to Theorem \ref{prop_2} there exists a sequence 
\begin{equation}\label{jaaa}
\{j_k\}_{k\geq1}\subset\bn\ \ \ \ s.t.\ \ P_{j_k\dots j_k,k}=1.
\end{equation} 

For sequence \eqref{jaaa} we set
$$
\mathcal D_r=r\cdot conv\left\{f_{j_k\dots j_k}\right\}_{k\in\bn},\ \ \ r\in[0,1].
$$
We notice that $\mathcal D_r\subset \mathcal D$ for any $0\leq r\leq1$ and the set
$\mathcal D'=\bigcup_{r\in[0,1]}\mathcal D_r$ is convex subset of $\mathcal D$.

We would like to show that operator ${\bf A}$ surjective on $\mathcal D'$.
First, we show the mapping $\Im$ given by \eqref{T} is a bijection from $\mathcal D_r$ to $S_r$.

Let $r\in[0,1]$. It is easy to check $\Im(x)\in S_r$ for any $x\in\mathcal D_r$.
We have $\Im(f_{j_i\dots j_i})=\eb_i$, $\forall i\in\bn$. Take any $\xb\in S_r$ with
$\xb=\sum_{i=1}^\infty x_i\eb_i$. Then we obtain
\begin{eqnarray*}
\Im\left(\sum_{n=1}^\infty x_n f_{j_nj_n\dots j_n}\right)_k&=&\int_\Omega\sum_{n=1}^\infty x_na_k(t)f_{j_nj_n\dots j_n}(t)dt\\
&=&\sum_{n=1}^\infty x_n\int_\Omega a_k(t)f_{j_nj_n\dots j_n}(t)dt\\
&=&\sum_{n=1}^\infty P_{j_nj_n\dots j_n,k}x_n\\
&=&x_k,
\end{eqnarray*}
which implies $\Im$ is surjective.\\
Now we take $x,y\in\mathcal D_r$. Then due to convexity of $\mathcal D_r$
there exist $(x_1,x_2,\dots)\in S$ and $(y_1,y_2,\dots)\in S$ such that
$$
x=r\sum_{n=1}^\infty x_n f_{j_nj_n\dots j_n},\ \ \ y=r\sum_{n=1}^\infty y_n f_{j_nj_n\dots j_n}.
$$
Assume that $\Im(x)=\Im(y)$. Then, from \eqref{T} and \eqref{jaaa} we immediately get $x_n=y_n$ for any $n\in\bn$.
Hence, $x=y$, which shows injectivity of $\Im$. Thus, we have shown that $\Im$ is surjective and injective mapping
from $\mathcal D_r$ to $S_r$. So, we infer that $\Im$ is bijection.

Now, we are ready to prove surjectivity of ${\bf A}$ on $\mathcal D'$. We notice that
\eqref{A12},\eqref{x_k} implies that ${\bf A}(\mathcal D_r)\subset\mathcal D_{r^m}$ for any $r\in[0,1]$.

Pick any $y\in\mathcal D'$. Without
loss of generality,  we may assume that $y\neq0$, i.e. $y\in\mathcal D_{r^m}$ for some $r\in(0,1]$.
Then $r^{-m}y\in\mathcal D_1$. According to Lemma \ref{lem2} we can find $\xb\in S$ such that $V(\xb)=r^{-m}\Im(y)$.
Since $r^mV(\xb)=V(r\xb)$ we infer that $V(r\xb)=\Im(y)$.
Due to the $\Im$ is bijection from $\mathcal D_r$ to $S_r$,
we can find a function $x\in\mathcal D_r$ such
that $\Im(x)=r\xb$.
On the other hand, \eqref{x_k} yields that $\Im {\bf A}=V\Im$. Keeping in mind this fact, one has
$$
\Im(y)=V(r\xb)=V(\Im(x))=\Im({\bf A}x).
$$
From $y,{\bf A}x\in\mathcal D_{r^m}$ and injectivity of $\Im$ on $\mathcal D_{r_m}$ one gets ${\bf A}x=y$. The arbitrariness of $y\in\mathcal D'$ implies that
$\mathcal D'\subset{\bf A}(\mathcal D')$. The last one together with ${\bf A}(\mathcal D')\subset\mathcal D'$
yield ${\bf A}(\mathcal D')=\mathcal D'$.
The proof is complete.
\end{proof}

\begin{proof}[Proof of Theorem \ref{thm_fixed}]
We notice that
$$
|Fix_{{\bf B}_1^+}(V)|=|Fix_{\mathcal D'}({\bf A})|,
$$
where $\mathcal D'$ is a set which constructed in the proof of Theorem \ref{thm_operA}.
Since $\mathcal D'\subset\mathcal D$ we immediately get
$$
|Fix_{{\bf B}_1^+}(V)|\leq|Fix_{\mathcal D}({\bf A})|.
$$
which completes the proof.
\end{proof}

Let us provide an example of the kernel which satisfies all the required conditions of
Theorem \ref{thm_operA}.

\begin{exm}\label{ex3} Let $\Omega=[0,1]$ with the usual Lebesgue measure. For the sake of convenience,
we assume that $m=3$. Let $K(u_1,u_2,u_3,t)$ be a function on $\Omega^{4}$ given by \eqref{yadro}, where $a_k(t)$ and $f_{i_1i_2i_3}(t)$ are
defined as follows:
$$
a_k(t)=\left\{
\begin{array}{lll}
2^k, & t\in(0,\frac{1}{2^{k+1}}),\\[3mm]
-2^k, & t\in(\frac{1}{2^{k+1}},\frac{1}{2^k}),\\[4mm]
0, & \mbox{otherwise}
\end{array}
\right.
$$
and
$$
f_{i_1i_2i_3}(t)=f_{i_{\pi(1)}i_{\pi(2)}i_{\pi(3)}}(t)=\left\{
\begin{array}{ll}
2^{-i_1-i_2-i_3}a_{i_1}(t)a_{i_2}(t)a_{i_3}(t), & \mbox{if }\ i_1\neq i_2\neq i_3\neq i_1;\\[3mm]
2^{-i_1-i_2}a_{i_1}(t)a_{i_2}(t), & \mbox{if }\ i_1\neq i_2=i_3;\\[3mm]
2^{-i_1}a_{i_1}(t), & \mbox{if }\ i_1=i_2=i_3;\\
\end{array}
\right.
$$
for any permutation $\pi$ of $(1,2,3)$. Then, we have $P_{i_1i_2i_3,k}=P_{i_{\pi(1)}i_{\pi(2)}i_{\pi(3)},k}$.
Furthermore, one can check
$$
P_{i_1i_2k,k}=\left\{
\begin{array}{ll}
1, & k\geq\max\{i_1,i_2\},\\
0, & k<\max\{i_1,i_2\},
\end{array}
\right.
$$
and
$$
P_{i_1i_2i_3,k}=0,\ \ \mbox{if }\ k\notin\{i_1,i_2,i_3\}.
$$
Then, $3$-ordered PSO $V$ given by
$$
\sum_{k=1}^m(V(\xb))_k=\left(\sum_{k=1}^m x_k\right)^3,\ \ m=1,2,3,\dots,\ \forall \xb\in S,
$$
satisfies all conditions of Theorem \ref{last_thm}. Hence, $V$ is surjective on $S$.
Correspondingly, integral operator \eqref{intOp} is surjective on $\mathcal D'=\cup_{r\in[0,1]}\mathcal D_r$, where
$$
\mathcal D_r=r\cdot conv\{f_{iii}\}_{i\geq1}.
$$
\end{exm}


\begin{thebibliography}{9}
	
	\bibitem{Has_Far_OPQSO_e} Akin H. and Mukhamedov F. Orthogonality preserving infinite dimensional quadratic stochastic
	operators, \textit{AIP Conf. Proc.} {\bf 1676}(2015), 020008.
	
	\bibitem{Arut2012} Artyunov A. V., Two problems of the theory of quadratic maps, \textit{Funct. Anal. Appl.} {\bf 46}(2012) 225--227
	
	\bibitem{1} Bernstein S. N. The solution of a mathematical problem
	concerning the theory of heredity,   \emph{Ann. Math. Statistics}
	{\bf 13}(1924) 53--61
	
	\bibitem{6} Ganikhodzhaev R., Mukhamedov F. and Rozikov U. Quadratic stochastic operators and processes: results and open problems, \emph{Infin.
		Dimens. Anal. Quant. Prob. Relat. Top.} {\bf 14}(2011) 279--335
	
	\bibitem{J2013} Jamilov U. U. Quadratic stochastic operators corresponding to graphs, \textit{Lobachevskii  J. Math} {\bf 34}(2013), 148--151
	
	\bibitem{Kol} Kolokoltsov V. N., \textit{Nonlinear Markov Processes and Kinetic Equations}, Cambridge Univ. Press, New York,
	2010.
	
	\bibitem{LZ}  Li C.-K., Zhang S., Stationary probability vectors of
	higher-order Markov chains, \textit{Linear Algebra Appl.} {\bf 473}
	(2015) 114-125.
	
	\bibitem{11} Lyubich Yu. I.  \textit{Mathematical structures in population
		genetics}, Berlin, Springer-Verlag, 1992.
	
	\bibitem{M2000} Mukhamedov  F. M.  On infinite dimensional Volterra operators,
	\textit{Russian Math. Surveys} {\bf 55}(2000), 1161--1162
	
	\bibitem{Far_Has_Temir} Mukhamedov  F., Akin H. and Temir S.  On infinite dimensional
	quadratic Volterra operators, \textit{J. Math. Anal. Appl.} {\bf
		310}(2005), 533--556.

	\bibitem{FOA} Mukhamedov  F., Khakimov O.N., Embong A.F. On Surjective second order non-linear Markov operators and associated nonlinear integral equations, {\it Positivity} DOI:10.1007/s11117-018-0587-0

	\bibitem{ME18}  Mukhamedov  F.,  Embong A.F. On non-linear Markov operators: surjectivity vs
	orthogonal preserving property,
	\textit{Lin. Multilin. Alg.} DOI: 10.1080/03081087.2017.1389849.
	
	\bibitem{MER17}  Mukhamedov  F.,  Embong A.F., Rosli A. Orthogonal preserving and surjective cubic stochastic operators,
	\textit{Ann Funct. Anal.} {\bf 8} (2017), 490-- 501.
	
	
	
	
	\bibitem{MG2015}  Mukhamedov F. and Ganikhodjaev N. \textit{Quantum Quadratic Operators and
		Processes}, Berlin, Springer, 2015
	
	\bibitem{taha}  Mukhamedov F. and Taha M. H., On Volterra and
	orthoganality preserving quadratic stochastic
	operators,  \textit{Miskloc Math. Notes} {\bf 17}(2016) 457--470
	
	\bibitem{Q} Qi L., Eigenvalues of a real supersymmetric tensor, \textit{J. Symbolic
		Comput.}, {\bf 40}(2005), 1302--1324.
	
	\bibitem{Raf} Raftery A., A model of high-order markov chains,
	\textit{J. Royal Statist. Soc.} {\bf 47}(1985), 528--539.
	
	\bibitem{man(2016)2d} Saburov M., On the surjectivity of quadratic stochastic operators acting on the simplex, \textit{Math. Notes} {\bf 99}(2016) 623-627
	
	
\end{thebibliography}
\end{document}